\documentclass[12pt,reqno]{amsart}
\usepackage[margin=2.5cm
]{geometry}
\usepackage{thmtools}
\usepackage{thm-restate}

\usepackage{calc}
\usepackage{mathtools}
\usepackage{amsaddr}

\makeatletter
\def\paragraph{\@startsection{paragraph}{4}%
  \z@\z@{-\fontdimen2\font}%
  {\normalfont\bfseries}}
\makeatother

\usepackage{graphicx}
\usepackage{enumerate}

\usepackage{amssymb, amsmath, amsthm}

\usepackage[hyperfootnotes=false, colorlinks, linkcolor={blue}, citecolor={magenta}, filecolor={blue}, urlcolor={blue}]{hyperref}

\usepackage{cleveref}

\allowdisplaybreaks[1]
\usepackage{nicefrac}

\newcommand{\ds}[1]{\displaystyle{#1}}

\newcommand{\RN}[1]{%
  \textup{\uppercase\expandafter{\romannumeral#1}}%
}

\newcommand*{\rom}[1]{\expandafter\romannumeral #1}

\allowdisplaybreaks

\newcommand{\p}{\mathfrak{p}}

\newcommand{\GammaON}{{\rm \Gamma_0}({\operatorname{N}})}

\newtheorem{thm}{Theorem}[section]
\newtheorem{lem}[thm]{Lemma}

\newtheorem{remark}[thm]{Remark}

\newcommand{\Eq}{}
\newcommand{\RR}{\mathbb{R}}      
\newcommand{\ZZ}{\mathbb{Z}}      

\newcommand{\mat}[4]{\left[\begin{smallmatrix*}
                #1 & #2 \\
                #3 & #4 \\
        \end{smallmatrix*}\right]}

\newcommand{\bmat}[4]{\left[\begin{matrix*}[r]
		#1 & #2 \\
		#3 & #4 \\
	\end{matrix*}\right]}

\newcommand{\para}{\mathrm{K}}

\newcommand{\Q}{\mathbb{Q}}

\newcommand{\Z}{\mathbb{Z}}

\newcommand{\mc}[1]{\mathcal{#1}}

\newcommand{\GSp}{\operatorname{GSp}}
\newcommand{\Sp}{\operatorname{Sp}}
\newcommand{\Si}[1]{{\rm Si}(p^{#1})}

\def\AA{{\mathbb A}}

\def\CC{{\mathbb C}}

\def\HH{{\mathbb H}}

\def\NN{{\mathbb N}}

\def\QQ{{\mathbb Q}}
\def\RR{{\mathbb R}}

\def\ZZ{{\mathbb Z}}

\def\ff{{\frak f}}

\DeclareMathOperator{\Kfun}{K}
\DeclareMathOperator{\Lfun}{L}

\DeclareMathOperator{\Sifun}{Si}

\newcommand{\Poincare}{Poincar\'e}

\DeclareMathOperator{\Sym}{Sym}

\def\<{\langle}
\def\>{\rangle}

\DeclareMathOperator{\GL}{GL}

\DeclareMathOperator{\SL}{SL}

\DeclareMathOperator{\K2}{K_p^2}

\newcommand{\invtr}{^{\text{-}\!\top}}

\newtheorem*{rmk*}{Remark}

\numberwithin{equation}{section}

\title {Pullback of Klingen Eisenstein series and certain critical L-values identities}
\author{Alok Shukla}
\address{University of Manitoba, Canada.}
\email{sajal.eee@gmail.com}

\curraddr{Alok Shukla \\
	Department of Mathematics\\
University of Manitoba\\
Winnipeg, Canada.}

\date{Aug 5, 2019.}
\thanks{\textit{Acknowledgements.} Author is grateful to Siddarth Sankaran for many helpful discussions. Author also thanks Ameya Pitale and Ralf Schmidt for suggesting this problem and for their helpful suggestions at the beginning of the project.}

\subjclass[2000]{Primary 11F46, 11F67}

\dedicatory{}

\keywords{Pullback of Klingen Eisenstein series, Critical L-values.}

\begin{document}
	
	\begin{abstract}
		We obtain pullback formulas for Klingen Eisenstein series with arbitrary levels, with respect to both Siegel congruence and paramodular subgroups, in degree two. Pullback results are used, along with the Fourier series expansion of Klingen Eisenstein series given by Mizumoto, to prove certain identities involving critical values of $ \Lfun $-functions attached to normalized elliptic modular forms of weight $ k $ and full level. 
	\end{abstract}
	
	\maketitle


\section{Introduction}
The idea of pullback of Eisenstein series is due to Garrett \cite{garrett1984pullbacks} who discovered it in the early 1980s. Thereafter, several authors have considered pullbacks in various settings [\cite{MR1166221}, \cite{heim1999pullbacks}, \cite{MR920329}, \cite{MR2380319}, \cite{MR2652263}, \cite{MR2806519}], and obtained a number of results related to the arithmetic and analytic properties of $ L $-functions. 

The aim of this paper is to obtain a pullback formula for Klingen Eisenstein series of an arbitrary level in degree two. We note that in \cite{heim1999pullbacks} pullbacks of Klingen Eisenstein series of degree $ n $ are considered. In contrast, here we restrict to  the degree two case, but we obtain pullback results for Klingen Eisenstein series of arbitrary levels. Indeed, for the full level case, i.e.~for $ N=1 $, our formula reduces to the degree two ($ n=1 $) case of the pullback formula of Heim in Theorem $ 2.3 $ in \cite{heim1999pullbacks}, and these results are in agreement. Our method is representation theoretic, whereas in \cite{heim1999pullbacks} the treatment is classical. Our representation theoretic methods allow us to consider pullbacks of Klingen Eisenstein series of level $ N $ with respect to both the Siegel congruence subgroup $ \Gamma_0(N) $ and the paramodular subgroup $ \Kfun(N) $ (see \eqref{Eq:defGammaN} and \eqref{KNdefeq} for definitions of $ \Gamma_0(N) $ and  $ \Kfun(N) $, respectively).
As a corollary of the pullback formula for $ N=1 $, and using the Fourier series expansion of Klingen Eisenstein series calculated by Mizumoto \cite{MR733590}, we prove certain identities containing critical values of $ \Lfun $-functions attached to normalized elliptic modular forms of weight $ k $ and full level.

\subsection{Klingen Eisentein series}
We recall some basic facts related to Klingen Eisenstein series.
\subsubsection{Classical formulation}
Klingen Eisenstein series are examples of holomorphic Siegel modular forms and were defined by Klingen in \cite{MR0219473}. A Klingen Eisenstein series $E_{n,r}^k(Z,f)\in M_k^{(n)}$ is obtained by lifting a cusp form $f\in S_k^{(r)}$ for $0\leq r<n$, using a natural summation.  We note that $ E_{n,r}^k(Z,f) $ is convergent if $k>n+r+1$. Several results related to Fourier coefficients and Hecke eigenvalues of Klingen Eisenstein series are known \cite{Bocherer1982, Bocherer1983, MR1087121, MR605300, MR636883, MR733590}. In this paper, we concentrate on the case $n=2$ and $r=1$, and investigate the pullback of Klingen Eisenstein series.

Let $ f $ be an elliptic cusp form of level $ N $, weight $ k $, i.e., $ f \in S_k^{(1)}(\Gamma_0^2(N)) $. A  Klingen Eisenstein series of level $ N $ with respect to $ \Gamma_0^4(N) $, is defined by
\begin{align}{\label{Klingendef_level}}
E_{2,1}^k(Z,f,N) &:= \sum_{\gamma \in \left({Q}(\QQ) \cap \Gamma_0^4(N) \right)\backslash   \, \Gamma_0^4(N)} f( \gamma  \langle Z \rangle^*)  \det(j(\gamma,Z))^{-k}.
\end{align}
Here,  $ Z \in \HH_2 = \{z \in M_2(\CC)   \, | \, ^tz = z,\text{Im} \, z > 0 \}$,   
for $  \Sp(4,\ZZ)  \ni \gamma = \bmat{A}{B}{C}{D},\,  \gamma \langle Z \rangle := (A Z +B) (CZ +D)^{-1},\,  j(\gamma,Z) := CZ +D,\,$ $ \gamma \langle Z \rangle ^* =  \tilde{\tau}\, \text{ for }   \gamma\langle Z \rangle = \bmat{\tilde{\tau}}{\tilde{z}}{\tilde{z}}{\tilde{\tau'}}$. 

If $ N=1 $, then the Eisenstein series defined in \eqref{Klingendef_level} reduces to the classical Klingen Eisenstein series 
 \begin{equation}{\label{Klingendef}}
E_{2,1}^{k}(Z,f) = \sum_{\gamma \in C_{2,1} \backslash  \Sp(4,\ZZ)} f(\gamma \langle Z \rangle^*) \det(j(\gamma,Z))^{-k},
\end{equation}
where $ C_{2,1} := \{  \left[ \begin{smallmatrix}
* &  & * & *\\
* & * & * & * \\
* &  & * & * \\
&  &  & *
\end{smallmatrix} \right]   \in \Sp(4,\ZZ) \}$. 
If $ k \geq 6 $ is an even integer, then the series defined in \eqref{Klingendef_level} and \eqref{Klingendef} are convergent.\\

Next, we note that a Klingen Eisenstein series with respect to the paramodular subgroup $ K(N) $ (see $ \eqref{KNdefeq}$) is defined as 
\begin{align}{\label{Klingendef_levelpara}}
\tilde{E}_{2,1}^k(Z,f,N) &:= \sum_{\gamma\in D(N)} \det(j(\gamma,Z))^{-k}f((\Lfun_N\gamma\langle Z\rangle)^*),
\end{align}
 where $D(N)$ is a set of representatives for $\Lfun_N^{-1}Q(\Q)\Lfun_N\cap\Kfun(N^2)\backslash\Kfun(N^2)$, with
 \begin{equation}\label{LNdefeq}
 \Lfun_N=\left[\begin{smallmatrix}1&N\\&1\\&&1\\&&-N&1\end{smallmatrix}\right].
 \end{equation}

\subsubsection{Adelic formulation}
The classical Klingen Eisenstein series of level $ N $ defined in \eqref{Klingendef_level} can be characterized by using a special Klingen induced global representation (see Theorem~$ 7.3.1 $ and Theorem~$ 7.3.2 $, \cite{shukla2018klingen}).
Indeed, we will use this representation theoretic formulation to study the pullback of classical Klingen Eisenstein series $ E_{2,1}^{k}(Z,f) $.

Let $ (\pi, V_{\pi}) $ be a cuspidal automorphic representation of $ \GL(2,\AA) $ and let $ \chi $ be a character of ideles. Then the standard notation $ \chi \rtimes  \pi $ will denote a family of induced representations of $ \GSp(4,\AA) $ that is globally induced via the normalized parabolic induction from  the Klingen subgroup $ Q(\AA) $  of $ \GSp(4,\AA) $. More explicitly, the space $ \chi \rtimes  \pi $  consists of functions $ \tilde{\phi} \colon \GSp(4,\AA) \rightarrow  V_{\pi}$ with the transformation property 
\begin{align} 
\tilde{\phi}(hg)& = |t^2 \, (ad-bc)^{-1}|\, \chi(t) \, \pi_1(\mat{a}{b}{c}{d}) \, \tilde{\phi}(g), \nonumber \text{ for all } h =   \left[ \begin{smallmatrix}
a &   & b & * \\
* & t & * & * \\
c &   & d & * \\
&   &   & t^{-1}(ad -bc)
\end{smallmatrix}  \right]     \in Q(\AA).
\end{align}
Suppose $ {\Phi}$ is a  vector in the global induced automorphic representation  $ |\cdot|^{s} \rtimes |\cdot|^{\frac{-s}{2}} \pi $ of $ \GSp(4,\AA) $.  For $ \tilde{g} \in \GSp(4,\AA) $, we define 
\begin{align}\label{def:KlingenAdelicPullback}
 {E}(\tilde{g},s,\Phi) :=  \sum_{\gamma \in {Q}(\QQ) \backslash \GSp(4,\QQ)} (\Phi (\gamma \tilde{g}))(1).
\end{align}
Now we fix the following embedding of $H_{1,1} := \{(g, g') \in \GSp(2) \times \GSp(2): \mu_(g) = \mu_(g') \}$ in $\GSp(4)$:
\begin{equation}\label{embedding-defn}
H_{1,1} \ni (\mat{a_1}{b_1}{c_1}{d_1}, \mat{a_2}{b_2}{c_2}{d_2}) \longmapsto \left[\begin{smallmatrix}a_1&&-b_1\\&a_2&&b_2\\-c_1&&d_1\\&c_2&&d_2\end{smallmatrix}\right] \in \GSp(4).
\end{equation}
Abusing notation, let $H_{1,1}$ also denote its image in $\GSp(4)$.

\subsection{Main results on Pullback of Klingen Eisenstein series}
We will make use of the Eisenstein series \begin{align}\label{def:Eisensteinseriesmodular}
E_1(s,k,N,\tau) := \sum_{\substack{    \mat{a}{b}{c}{d} \in  \Gamma_{\infty} \backslash \Gamma_0^2(N)    }}\; \; \frac{1}{|c \tau +d|^{s+2-k}(c \tau +d)^k}, 
\end{align}
which for $ s = k-2 $ is the classical holomorphic Eisenstein series of weight $ k $ with respect to $ \Gamma_0^2(N) $, up to a possible normalization. 
Now we give our main result on the pullback of Klingen Eisenstein series. 
\begin{restatable}[]{thm}{mainpullback}\label{thm:main_pullback}
Let $ S $ be a finite set of primes and $ N = \prod_{p \in S} p^{n_p} $ be a positive integer. Let $ f$ be an elliptic cusp forms of level $ N $, weight $ k $, i.e., $ f \in S_k^{(1)}(\Gamma_0^2(N)) $, with $ k > 4 $ and even. We also assume $ f$ to be a newform. Let $ \phi$ be the  automorphic form associated with $ f $ and  let $(\pi,V_{\pi})  $ be the irreducible cuspidal automorphic representation of $ \GL(2,\AA) $ generated by $ \phi $. Let  
$$g = (g_1,g_2) \in \GL(2,\AA),   \text{ with } g_{j} =  (g_{j,\infty},1,1,1,\cdots),  $$
where notation $ (\cdot\,,\cdot) $ is as explained after $ \eqref{embedding-defn} $, and $ g_{j,\infty} = \bmat{1}{x_j}{}{1}  \bmat{\sqrt{y_j}}{}{}{\sqrt{y_j}^{-1}}$, $ \tau_j = x_j + i y_j \in \HH$.	
Then there exists a global distinguished vector $ {\Phi}$ in the global induced automorphic representation  $ |\cdot|^{s} \rtimes |\cdot|^{\frac{-s}{2}} \pi $ of $ \GSp(4,\AA) $, for $ s = k-2 $, 	such that
\begin{align}\label{Eq:main_pullback_formula}
&E_{2,1}^k(\bmat{\tau_1}{}{}{\tau_2},f,N) = (y_1 y_2)^{-\frac{s+2}{2}}  {E}(g,s,\Phi) = E_1(s,k,N,\tau_1)f(\tau_2) + E_1(s,k,N,\tau_2) f(\tau_1) \nonumber \\ &+ 2 \sum_{\substack{c, d \, \in \NN, \\ (c,d)= 1}}   \sum_{\gamma_1,\gamma_2 \in  \Gamma_{\infty} \backslash \Gamma_0^2(N)}  j(\gamma_1,\tau_1)^{-(s+2)} j(\gamma_2,\tau_2)^{-(s+2)}  f(d^2 \gamma_1 \langle \tau_1 \rangle + c^2 \gamma_2  \langle \tau_2 \rangle),
\end{align}
where $ E_{2,1}^k(Z,f,N) $ is defined in \eqref{Klingendef_level}, $ {E}(g,s,\Phi) $ is defined in \eqref{def:KlingenAdelicPullback} and $ E_1(s,k,N,\tau) $ is as in \eqref{def:Eisensteinseriesmodular}.
\end{restatable}

The following result gives a pullback formula for Klingen Eisenstein series with respect to  the paramodular subgroup $ \Kfun(N) $.

\begin{restatable}[]{thm}{mainpullbackparamodular}\label{thm:main_pullbackparamodular}
	Let $ S $ be a finite set of primes and $ N = \prod_{p \in S} p^{n_p} $ be a positive integer. Assume $ \chi $ to be a Dirichlet character modulo $ N $. Also, $ \chi $ could be viewed as a continuous character of ideles, which we also denote by $ \chi $. Let $ f$ be an elliptic cusp form of level $ N $, weight $ k $ and character $ \chi $, i.e., $ f \in S_k^{(1)}(\Gamma_0^2(N),\chi) $, with $ k > 4 $ and even. We also assume $ f$ to be a newform. Let $ \phi$ be the  automorphic form associated with $ f $ and  let $(\pi,V_{\pi})  $ be the irreducible cuspidal automorphic representation of $ \GL(2,\AA) $ generated by $ \phi $. Let  
	$$
	g = (g_1,g_2) \in \GL(2,\AA),   \text{ with } g_{j} =  (g_{j,\infty},1,1,1,\cdots)
	,  $$
	where notation $ (\cdot\,,\cdot) $ is as explained after $ \eqref{embedding-defn} $, and $ g_{j,\infty} = \bmat{1}{x_j}{}{1}  \bmat{\sqrt{y_j}}{}{}{\sqrt{y_j}^{-1}}$, $ \tau_j = x_j + i y_j \in \HH$.	
	Then there exists a global distinguished vector $ {\Phi}$ in the global induced automorphic representation  $ \chi^{-1} |\cdot|^{s} \rtimes |\cdot|^{\frac{-s}{2}} \pi $ of $ \GSp(4,\AA) $, for $ s = k-2 $, 	such that
	\begin{align}
 &\tilde{E}_{2,1}^k(\bmat{\tau_1}{}{}{\tau_2},f,N) =	(y_1 y_2)^{-\frac{s+2}{2}} {E}(g,s,\Phi)  = E_1(s,k,N^2,\tau_1)f(\tau_2) + E_1(s,k,N^2,\tau_2) f(\tau_1) \nonumber \\
&+ 2 \sum_{\substack{c ,d \, \in \NN,  \\ (c,d)= 1}}  \sum_{\substack{\gamma_1 \in  \Gamma_{\infty} \backslash \SL(2,\ZZ), \\ \gamma_2 \in  \Gamma_{\infty} \backslash \Gamma_0^2(N^2)}}  j(\gamma_1,\tau_1)^{-(s+2)} j(\gamma_2,\tau_2)^{-(s+2)}  f(d^2 \gamma_1 \langle \tau_1 \rangle + c^2 \gamma_2  \langle \tau_2 \rangle),
	\end{align}
	where $ \tilde{E}_{2,1}^k(Z,f,N) $ is defined in $ \eqref{Klingendef_levelpara} $, $ {E}(g,s,\Phi) $ is defined in \eqref{def:KlingenAdelicPullback} and $ {E_1}(s,k,N,\tau) $ is as in \eqref{def:Eisensteinseriesmodular}.
\end{restatable}

Next we give an application of the pullback formula  \eqref{Eq:main_pullback_formula} in the special case of $ N=1 $.
\subsection{Main results on critical \texorpdfstring{$L$}{L}-values}
It can be safely stated that $ L $-functions are one of the central objects in modern number theory. Indeed, many important theorems and conjectures in number theory are related to $ L $-functions. Critical values of $ L$-functions hold arithmetically significant information and as such have been widely studied.  Dirichlet proved his famous result on primes in arithmetic progressions using the non-vanishing of Dirichlet $ L $-functions, i.e., $ L(\chi,1) \neq 0 $ for any non-principal Dirichlet character $ \chi $ of period $ N $. For more modern examples, one can refer to Birch and Swinnerton-Dyer \cite{MR0230727}, Bloch-Beilinson conjecture [\cite{MR575206}, \cite{MR944989}] and Bloch-Kato conjecture \cite{MR1086888}, which are some of the famous conjectures involving $ L $-functions. 

In the following we describe our identities involving certain critical $ L $-values.

\subsection{Preliminaries}
We denote the Rankin convolution $L$-function of modular forms $f$ and $g$ by $L(s,f \otimes g)$. Also, let $L(s,\Sym^2 f)$ denote the symmetric square $L$-function of $f$.  
For $T = \mat{a}{b/2}{b/2}{c}$,  we define the theta function
\begin{align*}
\vartheta_T(z) &:= \vartheta_{a,b,c}(z) := \sum_{(m,n)\in \mathbb{Z}^2} q^{am^2+bmn + cn^2}.
\end{align*}
For notational convenience, we set $$ \vartheta_1(z) :=  \vartheta_{1,0,1}(z) =\sum_{(m,n)\in \mathbb{Z}^2} q^{m^2+n^2},$$ 
and $$\vartheta_2(z) :=  \vartheta_{1,1,1}(z) = \sum_{(m,n)\in \mathbb{Z}^2}q^{m^2+mn+n^2}.$$
For $T$ such that $-\det(2T)$ is a fundamental discriminant, let $\chi_{-\det(2T)}$ be the Dirichlet character associated to the field $\mathbb{Q}(\sqrt{-\det(2T)})$. 
Let $ f $ be a normalized elliptic cusp form of weight $ k $, i.e., $ f \in S_{k}^{(1)}  $. Let
\begin{equation}
F_{f,c,d}(\tau_1,\tau_2) = \sum_{\gamma_1, \gamma_2 \in \Gamma_{\infty} \backslash \SL(2,\ZZ)} j(\gamma_1,\tau_1)^{-k} j(\gamma_2,\tau_2)^{-k} f\left( c^2 \gamma_2 \langle\tau_2 \rangle + d^2 \gamma_1 \langle \tau_1 \rangle \right),
\end{equation}
where $ \tau_j  $ are elements of \Poincare{} upper half-plane $ \HH $, $ c $ and $ d $ are co-prime integers, and $ \Gamma_\infty = \{ \pm \bmat{1}{m}{}{1} \,|\, m \in \ZZ  \} $ is the stabilizer of the cusp at infinity. Suppose $ F_{f,c,d}(\tau_1,\tau_2) $ has the following Fourier expansion
\begin{equation}
F_{f,c,d}(\tau_1,\tau_2) = \sum a_{c,d}(n_1,n_2) q_1^{n_1} q_2^{n_2},
\end{equation}
where $ q_j = e^{i 2\pi \tau_j} $. We define
\begin{equation} \label{Eq:defA_f}
A_f(n_1,n_2) := 2 \sum_{\substack{ c ,d \, \in \NN, (c,d) =1 }}  a_{c,d}(n_1,n_2).
\end{equation}

\begin{restatable}[\textbf{of Theorem \ref{thm:main_pullback}, and Mizumoto, Theorem 1 \cite{MR636883}}]{cor}{mainidentity}
	\label{thm:main_identity}
	Let $ f $ be a normalized elliptic cusp form of weight $ k $, i.e., $ f \in S_{k}^{(1)}  $. Then we have the following identity.
	\begin{align} \label{eq:theorem1.1}
	\frac{4}{\zeta(1-k)} + A_f(1,1) = 2+\frac{(-1)^{k/2}(k-1)!(2\pi)^{k-1}}{(2k-2)!L(2k-2,\Sym^2 f)}  \left[2^{2k-3}L(k-1,\chi_{-4})L(k-1,f \otimes \vartheta_1) \right. \nonumber \\ \left. + 2 \cdot 3^{k-3/2}L(k-1,\chi_{-3})L(k-1,f \otimes \vartheta_2)\right],
	\end{align}
	with $ A_f(1,1) $ given by $ \eqref{Eq:defA_f} $.	
\end{restatable}

Now we state another result obtained as an application of the pullback formula for Klingen Eisenstein series.
\begin{restatable}[\textbf{of Theorem \ref{thm:main_pullback}, and Mizumoto, Theorem 1 \cite{MR733590}}]{cor}{maintheorem}\label{thm:main_theorem}
	Let $ f $ be a normalized cuspidal eigenform of weight $ k $ with Fourier expansion $\ds f(z) = \sum_{n=1}^{\infty} a(n) q^n$. For non-negative numbers $ n_1,n_2 $, let
	\begin{align}
	\Lambda (n_1,n_2) := \left\{ \left(\begin{smallmatrix} n_1& b/2 \\ b/2 &n_2 \end{smallmatrix} \right)  \vert \, b \in \ZZ, 4n_1 n_2 - b^2 \geq 0 \right\}.
	\end{align}
	For $ T \in \Lambda (n_1,n_2) $, let the positive integer $ \ff_{\scriptscriptstyle{T}} $ be defined as $\ds  \ff_{\scriptscriptstyle{T}} := \sqrt{\frac{\det{(2T)}}{\Delta(T)}} $, where 
	$ -\Delta(T) $ is the discriminant of the quadratic field $ \QQ(\sqrt{- 2 \det(T)}) $. 
	If $ \vartheta_T(z) $ has the Fourier expansion $ \vartheta_T(z) = \sum_{n=1}^{\infty} b_T(n) q^n $, then let the $ v^2 $-twisted Rankin-Selberg $ L $-function of $ f $ and $ \vartheta_T $ be defined as 
	\[
	L(s,f \otimes^{\prime} \vartheta_T^{(v)} ) = \sum_{n=1}^{\infty} a(n) b_T(v^2 n) n^{-s}.
	\]
	Let $ \gcd(n_1,n_2)=1 $. Then 	\begin{align*}
	& \sum_{T \in \Lambda(n_1,n_2)} (-1)^{\frac{k}{2}} \Delta(T)^{k- \frac{3}{2}}  \left( \frac{L(k-1, \chi_{-\det(2T)}) }{L(2k-2,f,\Sym^2)}\right) \left( \sum_{v |\ff_{\scriptscriptstyle{T}}, v > 0} L(k-1,f \otimes^{\prime}  \phi_{T,v}) \right) \\ &= \frac{ ((2k-2)!)}{(2 \pi)^{k-1} (k-1)! }  \left( \frac{2}{\zeta(1-k)} \left[a(n_1) \sigma_{k-1}(n_2) + a(n_2) \sigma_{k-1}(n_1) \right] + A_f(n_1,n_2) \right),
	\end{align*}
	where 
	$$\phi_{T,v} =   (\ff_{\scriptscriptstyle{T}} v^{-1})^{2k-3} \prod_{\substack{p | \ff_{\scriptscriptstyle{T}} v^{-1}\\ p :\text{ prime}}} \, \left(1 - p^{1-k} \, \chi_{-\det(2T)} (p)\right) \vartheta_T^{(v)},$$ and $\ds \sigma_k(n) := \sum_{d |n}  d^k$ is the usual divisor function.	
\end{restatable}

\begin{remark}\label{remark:cor}
	\leavevmode
	\begin{enumerate}
		\item 	The proofs of  \Cref{thm:main_identity}  and  \Cref{thm:main_theorem}  depend on pullback formula for Klingen Eisenstein series  given in Theorem~$\ref{thm:main_pullback} $ for $ N=1 $, and the Fourier series expansion of Klingen Eisenstein series calculated by Mizumoto \cite{MR733590}. 
		\item  For $ N=1 $, the pullback formula in Theorem~$\ref{thm:main_pullback}$ is not a new result. In fact, for $ N=1 $ it reduces to the degree two ($ n=1 $) case of the pullback formula of Heim in Theorem $ 2.3 $ in \cite{heim1999pullbacks}, which is essentially due to Garrett \cite{garrett1987decomposition}. 
		\item 	One can obtain more identities in the spirit of \Cref{thm:main_identity}  and  \Cref{thm:main_theorem}, by using $ (1.3) $ and Theorem~$ 1 $ in \cite{MR733590}.
		\item  Mizumoto presented the Fourier series expansion of Klingen Eisenstein series in \cite{MR733590} for the full level case, i.e., $ N=1 $.   As of now, to the best of our knowledge, finding the Fourier	coefficients of Klingen Eisenstein series with respect to $ \GammaON $ for $ N > 1 $, is still an open problem. If such a formula for the Fourier	coefficients of Klingen Eisenstein series with respect to $ \GammaON $ for $ N>1 $, becomes available,  then  in conjunction with our pullback results, immediately new and interesting identities involving critical $ L $-values could be derived similar to the proof of \Cref{thm:main_identity}  and  \Cref{thm:main_theorem}.
	\end{enumerate}
\end{remark}

After describing our main results, now we give a brief organization of the rest of the paper. In Sect.~$ \ref{sec:notation} $ we fix our notation. We will present our main result on pullback of Klingen Eisenstein series with respect to Siegel congruence subgroup in Sect.~$ \ref{Sect:double_coset} $. We will describe pullback of Klingen Eisenstein series with respect to Paramodular subgroup in Sect.~$ \ref{Sect:paramodular} $. Finally in Sect.~$ \ref{Sect:applications} $ we will give some applications of the pullback formula.

\section{Notation} \label{sec:notation}
We realize the group $ \GSp(2n) $ as 
$$\GSp(2n) := \{g \in \GL(2n) \ | \  ^tgJg = \lambda(g) J  \text{ for some } \lambda(g) \in \GL(1) \}, $$  with   $ J = \bmat{}{I_n}{-I_n}{} $. Siegel half space of degree (or genus) $ n $ will be denoted by $ \HH_n := \{z \in M_{n}(\CC)  \, | \, ^tz = z,\text{Im} \, z > 0 \} $. Let $ B(2n) $ be the Borel (minimal parabolic) subgroup of $ \GSp(2n) $ and let $ N(2n) $ be its unipotent subgroup. Let $ Q(\QQ) $ be the Klingen parabolic subgroup of $ \GSp(4,\QQ) $ consisting of the matrices of the form 
\begin{equation}
 \{  \left[ \begin{matrix}
 * &  & * & *\\
 * & * & * & * \\
 * &  & * & * \\
 &  &  & *
 \end{matrix} \right] \ | \ * \in \QQ \}.
\end{equation}
We will denote the Siegel congruence subgroup of level $ N $ as 
\begin{equation}\label{Eq:defGammaN}
\Gamma_{0}(N) = \Gamma_{0}^{4}(N):= \{ \left[\begin{matrix*}[r] * & * &* &* \\ * &* & * &* \\a & b & *& *\\ c & d  & * & * \end{matrix*}\right ] \in \Sp(4,\ZZ)  \ | \ a,b,c,d \equiv 0 \mod{N} \}.
\end{equation}
Let 
\begin{equation}
 \Gamma_{0}^2(N) := \{ \bmat{a}{b}{c}{d} \in \SL(2,\ZZ) \ | \ c\equiv 0 \mod{N} \}
\end{equation}
 be the Hecke congruence subgroup of $ \SL(2,\ZZ) $. Let 
 \begin{equation}
 \Gamma_{\infty} := \Gamma^2_{\infty}(\ZZ) := \{\pm \bmat{1}{m}{}{1} |\, m \in \ZZ \}.
 \end{equation}
We will denote the local Siegel congruence subgroup of $ \GSp(4,\QQ_p) $ of level $ p^n $ by 
\begin{equation}
\Si{n} := \{\alpha = \left[\begin{matrix*}[r] * & * &* &* \\ * &* & * &* \\a & b & *& *\\ c & d  & * & * \end{matrix*}\right ] \in \GSp(4,\ZZ_p) \ | \ a,b,c,d \in p^n \ZZ_p \}.
\end{equation}
We denote { paramodular subgroup} of level $N$ as
\begin{equation}\label{KNdefeq}
\para(N) :=\Sp(4,\Z)\cap\left[ \begin{smallmatrix}
\Z&N\Z&\Z&\Z\\\Z&\Z&\Z&N^{-1}\Z\\\Z&N\Z&\Z&\Z\\N\Z&N\Z&N\Z&\Z\end{smallmatrix}  \right].
\end{equation}
The local version of the paramodular group is defined as 
\begin{equation}\label{Kpnlocaleq}
\para_p(p^n) :=\{g\in\GSp(4,\Q_p)\:|\:\lambda(g)\in\Z_p^\times,\:g\in\left[ \begin{smallmatrix}
\Z_p&p^n\Z_p&\Z_p&\Z_p\\\Z_p&\Z_p&\Z_p&p^{-n}\Z_p\\\Z_p&p^n\Z_p&\Z_p&\Z_p\\p^n\Z_p&p^n\Z_p&p^n\Z_p&\Z_p\end{smallmatrix}  \right]\}.
\end{equation}
We denote by $ L_N $ the following matrix
\begin{equation}
L_N := \left[ \begin{matrix}
1 & N &  & \\
 &  1&  &  \\
 &  &1 &  \\
&  & -N & 1
\end{matrix} \right].
\end{equation}

\section{Pullback of Klingen Eisenstein series with respect to Siegel congruence subgroup} \label{Sect:double_coset}
We note a double coset decomposition which will be useful later.

\begin{lem}
	A complete and minimal system of representatives for the double coset decomposition $ Q(\QQ) \backslash \GSp(4,\QQ)  / H_{1,1}(\QQ) $ is given by \\
	\begin{align*}
	1 ,\qquad s_1 = \left[\begin{smallmatrix*}[r]
	& 1 &  &  \\
	1 &  &  &  \\
	&  &  & 1 \\
	&  & 1 &
	\end{smallmatrix*}\right], \qquad   r = 
	\left[\begin{smallmatrix*}[r]
	1 &  &  &  \\
	1 & 1 &  &  \\
	&  & 1 & -1 \\
	&  &  & 1
	\end{smallmatrix*}\right].
	\end{align*}
	
\end{lem}
\begin{proof}
	The result is proved in Prop.~$ 2.4 $, \cite{garrett1987decomposition}. 
\end{proof}
Now by calculating $ (\eta^{-1} Q(\QQ) \eta \cap H_{1,1}(\QQ)) \backslash H_{1,1 }(\QQ)$, for each  $ \eta \in \{1,s_1,r\} $, the representatives of $ Q(\QQ) \backslash \GSp(4,\QQ) $ can be obtained. By multiplying with appropriate elements of $Q(\QQ)$ from left, the representatives over $ \ZZ $ can also be obtained. We have the following result (see Cor.~$ 2.2 $, \cite{heim1999pullbacks}).

\begin{lem}\label{Lem:Heim}
		The left cosets space $ Q(\QQ) \backslash \GSp(4,\QQ) $ has the following decomposition
		\begin{align}
		Q(\QQ) \backslash \GSp(4,\QQ) =  & \bigsqcup_{\substack{ \gamma \in \Gamma_{\infty} \backslash \SL(2,\ZZ )}}  (1,\gamma) \bigsqcup_{\substack{ \gamma \in \Gamma_{\infty} \backslash \SL(2,\ZZ )}}  s_1 (\gamma,1) \nonumber 
		 \bigsqcup_{\substack{ \gamma_1, \gamma_2 \in \Gamma_{\infty} \backslash \SL(2,\ZZ ) \\  c \in \NN, d \in \ZZ\backslash \{0\} \\ (c,d) =1}}  \epsilon_{c,d} (\gamma_1,\gamma_2),
		\end{align}
		where $ \epsilon_{c,d} =  \left[ \begin{matrix*}[r]
		d &  -c &  &  \\
		* & *  &  &  \\
		&   & * & * \\
		& &  c   & d
		\end{matrix*} \right] $.
\end{lem}
For convenience we will denote the three terms in the  Lemma \ref{Lem:Heim} as $ T_1 $, $ T_{s_1} $ and $ T_r $ \\
	Before  proceeding further, we note a simple matrix identity which will be useful later.
	\begin{equation}{\label{matrix_identity}}
	\bmat{a}{b}{c}{d}\bmat{1}{x_j}{}{1} \bmat{\sqrt{y_j}}{}{}{\sqrt{y_j}^{-1}} = \bmat{1}{x_j^{\prime}}{}{1} \bmat{\sqrt{y_j^{\prime}}}{}{}{\sqrt{y_j^{\prime}}^{-1}} r(\theta_j) 
	\end{equation}
	where, 
	\begin{equation*}
	\bmat{a}{b}{c}{d} \in \SL(2,\ZZ), \, \tau_j = x_j+iy_j, \, x_j^{\prime} + iy_j^{\prime} = \frac{a \tau_j + b}{c \tau_j + d},  \text{ and } \exp(i \theta_j) = \frac{c \bar{\tau_j} + d }{|c \tau_j + d|}.
	\end{equation*}
	We also note an explicit Iwasawa decomposition for $\bmat{a}{b}{c}{d} \in  SL(2,\RR) $ for a later use. We have
	\begin{equation}
	\bmat{a}{b}{c}{d} = \bmat{1}{ q}{}{1} \bmat{r^{-1}}{}{ }{r} \bmat{\cos \theta}{-\sin \theta}{\sin \theta}{\cos \theta} 
	\end{equation}
	where,
	\begin{align*}
	r^2 &= c^2 + d^2, r \sin \theta =c , r \cos \theta = d, 
	q = \frac{ac+bd}{c^2 + d^2}.
	\end{align*}
	
Now we prove our main theorem on pullback of Klingen Eisenstein series of level $ N $ with respect to Siegel congruence subgroup. 
\mainpullback*

\begin{proof}
	First of all we will explicitly define the global distinguished vector $ {\Phi} $.  
	Let $ \Pi $ denote the automorphic representation  $  |\cdot|^{s} \rtimes |\cdot|^{\frac{-s}{2}} \pi $  of 
	$ \GSp(4,\AA) $. We know from the tensor product theorem that  
	\begin{align}
	\Pi \cong \bigotimes_{p\leq \infty} \Pi_p
	\end{align}
	where almost all $ \Pi_p $ are unramified. For each prime $ p $, we select a local distinguished vector
	$ \Phi_p \in \Pi_p $, which would then yield a global distinguished vector 
	\begin{equation}{\label{defglobaldistinguishedvectorpara1}}
	{\Phi} \cong {\Phi}_\infty \otimes \bigotimes_{p < \infty}  {\Phi}_p   \, .
	\end{equation}
	Similarly, it follows from the tensor product theorem that
	\begin{align*}
	\phi \cong {\phi}_\infty \otimes \bigotimes_{p < \infty}  {\phi}_p. 
	\end{align*}
	Here, since $ \phi$ is the adelic cusp form associated with $ f \in S_k^1(\Gamma_0^2(N))$, it is clear that for every finite prime $ {p} $ with $ p \nmid N $, $ \phi_{p}$ is a spherical vector and for each prime $ {p} $ with $ p | N $, $ \phi_{{p}} $ is a $$\K2({p}^{n_p}) := \left \{  \bmat{a}{b}{c}{d} \in \GL(2,\ZZ_{p}) \colon c \in p^{n_p} \ZZ_{p},\, d \in 1 + {p}^{n_p} \ZZ_{p} \right\} $$ invariant vector. 
	Next, we describe our choices for the local distinguished vectors. \\ 
	{\textbf{Archimedean distinguished vector}}\\ 
	We pick a distinguished vector ${\Phi}_{\infty} $ such that 
	it is of the minimal $ K $-type $ (k,k) $. It was shown in Proposition $ 5.2.7 $, \cite{shukla2018klingen}, (which is a special case of Theorem $ 10.2 $, \cite{MR2587308}), that such a vector exists in $ \Pi_{\infty} $.
	More explicitly we define  $ \Phi_{\infty} $ as follows
	\begin{equation}{\label{defarchimedeandistinguishedvectorpara2}}
	{\Phi}_{\infty}(h_{\infty} k_{\infty}) :=  \det(j(k_\infty,I))^{-k} |t^2(ad -bc)^{-1}|_{} \,|t|_{}^{s}\, |ad -bc|_{}^{-\frac{s}{2}} \pi_{1,\infty}(\bmat{a}{b}{c}{d}) \,\phi_{\infty}
	\end{equation}
	where $$ h_{\infty} = \left[ \begin{matrix*}[r]
	a &   & b & * \\
	* &t & * & * \\
	c  &   & d & * \\
	& &   & t^{-1}(ad -bc)
	\end{matrix*}  \right] \in Q(\RR), $$ and $ k_{\infty}  \in \Kfun_1$, the maximal standard compact subgroup of $ \GSp(4,\RR) $. It can be checked that ${\Phi}_{\infty} $ is well-defined.\\ 
	\noindent \textbf{Unramified non-archimedean distinguished vectors} \\ 
	For all primes $ q $ such that $ q \nmid N $, we pick unramified local distinguished vectors such that,
	\begin{equation}{\label{defglobaldistinguishedvectorpara2}}
	{\Phi}_q(1) := \phi_q \, .
	\end{equation}    
	This means we have 
	\begin{equation}{\label{defglobaldistinguishedvectorpara2_new}}
	{\Phi}_q(g) = {\Phi}_q(h_q k_q) = |t^2(ad -bc)^{-1}|_{q} \,|t|_{q}^{s}\, |ad -bc|_{q}^{-\frac{s}{2}} \pi_{q}(\bmat{a}{b}{c}{d}) \phi_q, 
	\end{equation} 
	where using the Iwasawa decomposition $ g \in  \GSp(4,\QQ_q)$ is written as $ g = h_q k_q $  with 
	$$ h_q = \left[ \begin{matrix*}[r]
	a &   & b & * \\
	* &t & * & * \\
	c  &   & d & * \\
	& &   & t^{-1}(ad -bc)
	\end{matrix*}  \right] \in Q(\QQ_q) $$ and $ k_q  \in \GSp(4,\ZZ_q) $. It is easy to check that $ {\Phi}_q $ is well-defined.\\ 
	\noindent \textbf{Ramified non-archimedean distinguished vector} \\ 
	For each finite prime $ p $ such that $ p | N $ we select a $ \Gamma_0^4(p^{n_p}) $ invariant vector as distinguished vector.  We note that the existence of such a vector that is supported only on $ Q(\QQ_{p})1 \Sifun(p^{n_p})  $ follows from $ \Eq(7.2) $ and the discussion preceding that in \cite{shukla2018klingen}. Therefore, our distinguished vector $ \Phi_p $  is zero on all double cosets other than $Q(\QQ_{p})1 \Sifun(p^{n_p}) $ and on $Q(\QQ_{p})1 \Sifun(p^{n_p}) $ it is given by
	\begin{align}{\label{DeflocalcongruencedistinguishedvectorPrincipalSeries}}
	{\Phi}_{p}  \left( \left[ \begin{matrix*}[r]
	a &   & b & * \\
	* &t & * & * \\
	c  &   & d & * \\
	& &   & t^{-1}(ad -bc)
	\end{matrix*}  \right] 1 \,  \kappa   \right)  = &|t^2(ad -bc)^{-1}|_{p} \,|t|_{p}^{s}\,  |ad -bc|_{p}^{-\frac{s}{2}} \pi_{p}(\bmat{a}{b}{c}{d}) \phi_{p}
	\end{align}  where, $ \kappa \in \Sifun({p^{n_p}})$,$ \bmat{a}{b}{c}{d} \in \GL(2,\QQ_{p}) $  and  $ \phi_{p} \in   \pi_{p}$ is such that 
	$$ \pi_{p}(\bmat{a}{b}{c}{d}) \phi_{p} =  \phi_{p}  \text{ for all }  \bmat{a}{b}{c}{d} \in \bmat{\ZZ_{p}^{\times}}{\ZZ_{p}}{p^{n_p} \ZZ_{p}}{\ZZ_{p}^{\times}}. $$ 
	
	Now we compute the contribution of the three terms $ T_1 $, $ T_{s_1} $ and $ T_r $ in Lemma \ref{Lem:Heim}. \\

\underline{Contribution of $ T_1 $:} The contribution of $ T_1 $ term is given by
	
	\begin{align*}
I(T_1) = \sum_{\substack{\gamma_1 = \, (1,\gamma), \, \gamma \in \Gamma_{\infty} \backslash \SL(2,\ZZ )}}\; (\Phi (\gamma_1 g,s))(1)
	= \sum_{\substack{  \gamma \in\Gamma_{\infty}  \backslash \SL(2,\ZZ )}}\; \; (\Phi (  ( g_1,\gamma g_2),s))(1).
	\end{align*}
	Now, we have
	\begin{align*}
I(T_1) &= \sum_{\substack{  \gamma \in \Gamma_{\infty}  \backslash \SL(2,\ZZ )   }}\; \; (\Phi (  ( g_1,\gamma g_2),s))(1)\\
	& = \sum_{\substack{  \gamma \in \Gamma_{\infty}  \backslash \SL(2,\ZZ)   }}\; \;(|\det(g_1)|^{-\frac{s+2}{2}}  \pi(g_1) \Phi ( (1,(\gamma g_{2,\infty},\gamma,\cdots) ),s))(1) \\
	& = \sum_{\substack{  \gamma  \in \Gamma_{\infty}  \backslash \SL(2,\ZZ)   }}\; \pi(g_1) \
 \left( \Phi_{\infty} ((1,\gamma g_{2,\infty}),s)  \bigotimes_{p | N}  \Phi_p((1,\gamma),s) \bigotimes_{q < \infty, \, q \nmid   N} \Phi_q((1,\gamma),s)\right)(1).
	\end{align*}
	Since $ \Phi_p $ is supported only on $ Q(\QQ_p) 1 \Si{n_p} $ for each $ p |N $, it follows from Lemma 7.2, \cite{schmidt2018klingen}, which essentially depends on the double coset decompositions of $ Q(\QQ) \backslash \GSp(4,\QQ) / \Gamma_0(p^{n_p}) $ as computed in \cite{shukla2017codimensions}, that only those $ \gamma $ will contribute in the summation for which 
	$$ (1,\gamma) \in \Q(\QQ) \backslash   \bigcap_{p | N}  Q(\QQ) \, \Gamma_0^4({p^{n_p}}) = \Q(\QQ) \backslash   Q(\QQ) \, \Gamma_0^4(N) = (Q(\QQ) \cap \Gamma_0^4(N)) \backslash \Gamma_0^4(N) .$$ 
	Here we note that the first equality above is perhaps a little more subtle than it might appear at first (for a proof see Proposition~$ 7.2.1 $, \cite{shukla2018klingen}).
	It follows that in the above summation one can restrict to $ \gamma \in \Gamma_{\infty}  \backslash \Gamma_0^2(N) $. 
	
	Further, if $ q \nmid N $ then $ \Phi_q((1,\gamma),s) = \Phi_q(1,s) = \phi_q $, as  $ \Phi_q $ is unramified and $ (1,\gamma) \in \GSp(4, \ZZ_q ) $.  Also, if $ p|N $ then from the definition of $ \Phi_p $, we have $ \Phi_p((1,\gamma),s) = \Phi_p(1,s) = \phi_p  $. Therefore, 
	\begin{align*}
	I(T_1)  & = \sum_{\substack{  \gamma  = \mat{a}{b}{c}{d} \in  \Gamma_{\infty} \backslash \Gamma_0^2(N) }}\; \;  \pi(g_1) 
	\left( \Phi_{\infty} ((1,\gamma g_{2,\infty}),s)  \bigotimes_{p < \infty}  \phi_p \right)(1) \\
	& = \sum_{\substack{  \mat{a}{b}{c}{d} \in \Gamma_{\infty}  \backslash \Gamma_0^2(N) }}\; \; {y_2^{\prime}}^{\frac{s+2}{2}} e^{ik \theta_2} \pi(g_1) \phi(1) \\
	& = \sum_{\substack{  \mat{a}{b}{c}{d} \in \Gamma_{\infty}  \backslash \Gamma_0^2(N) }}\; \; {y_2^{\prime}}^{\frac{s+2}{2}} e^{ik \theta_2} \,  \phi(g_1)  \\
	& = \sum_{\substack{ \mat{a}{b}{c}{d} \in \Gamma_{\infty}  \backslash \Gamma_0^2(N)  }}\; \; \frac{y_2^{\frac{s+2}{2}}}{|c_2 \tau +d_2|^{s+2}} \left(\frac{c_2 \bar{\tau} + d_2 }{|c_2 \tau_2 + d_2|}\right)^k  \,  \phi(g_1)   \\ \\
	& \qquad \qquad \qquad \text{(we have used $ \eqref{matrix_identity} $ in the previous step)} \\ \\
	& = \sum_{\substack{  \mat{a}{b}{c}{d} \in \Gamma_{\infty} \backslash \Gamma_0^2(N) }}\; \; \frac{y_2^{\frac{s+2}{2}}}{|c_2 \tau +d_2|^{s+2-k}(c_2 \tau +d_2)^k}  \,  \phi(g_1) \\
	& = (y_1 y_2)^{\frac{s+2}{2}} E_1(s,k,N,\tau_2) \,  f(\tau_1). 
	\end{align*}
\underline{Contribution of $ T_{s_1} $:} A calculation similar to the previous one shows that the contribution of  $ T_{s_1} $ term is given by
\begin{align*}
I(T_{s_1})	& = (y_1 y_2)^{\frac{s+2}{2}} E_1(s,k,N,\tau_1) \,  f(\tau_2). 
\end{align*}
\underline{Contribution of $ T_{r} $:} Now, it only remains to find the contribution of $ T_{r} $.
\begin{align*}
& I(T_r) = \sum_{\substack{ \gamma_1, \gamma_2 \in \Gamma_{\infty} \backslash \SL(2,\ZZ ) \\  c \in \NN, d \in \ZZ\backslash \{0\} \\ (c,d) =1}} (\Phi (\epsilon_{c,d}  (\gamma_1 g_1,\gamma_2 g_2),s))(1) \\
	& = \sum_{\substack{  \substack{ \gamma_1, \gamma_2 \in \Gamma_{\infty} \backslash \SL(2,\ZZ ) \\  c \in \NN, d \in \ZZ\backslash \{0\} \\ (c,d) =1}}}\; 
	\left( \Phi_{\infty} (\epsilon_{c,d} (\gamma_1 g_{1,\infty},\gamma_2 g_{2,\infty}),s)  \bigotimes_{p | N}  \Phi_p((\gamma_1,\gamma_2),s) \bigotimes_{q < \infty, \, q \nmid   N} \Phi_q((\gamma_1,\gamma_2),s)\right)(1).
	\end{align*}
	Since $ \Phi_p $ is supported only on $ Q(\QQ_p) 1 \Si{n_p} $, we conclude, as we did earlier while determining the contribution of $ T_1 $,  that only those $ (\gamma_1,\gamma_2 )$ will contribute in the summation for which 
	$$ (\gamma_1,\gamma_2) \in \Q(\QQ) \backslash   \bigcap_{p | N}  Q(\QQ) \, \Gamma_0^4({p^{n_p}}) = \Q(\QQ) \backslash   Q(\QQ) \, \Gamma_0^4(N) = (Q(\QQ) \cap \Gamma_0^4(N)) \backslash \Gamma_0^4(N) .$$	
		
	Therefore, in the above summation one can restrict to summing over all  $ \gamma_1,\gamma_2 \in \Gamma_{\infty}  \backslash \Gamma_0^2(N) $.
		Next, if $ q \nmid N $ then $ \Phi_q((\gamma_1,\gamma_2),s) = \Phi_q(1,s) = \phi_q $, as  $ \Phi_q $ is unramified and $ (\gamma_1,\gamma_2) \in \GSp(4, \ZZ_q ) $.  Also from the definition of $ \Phi_p $, we have $ \Phi_p((\gamma_1,\gamma_2),s) = \Phi_p(1,s) = \phi_p  $. Therefore, 
\begin{align*}
		I(T_r)  & = \sum_{\substack{ c \in \NN, d \in \ZZ\backslash \{0\} \\ (c,d) =1}} \sum_{\substack{  \gamma_i  = \mat{a_i}{b_i}{c_i}{d_i} \in  \Gamma_{\infty} \backslash \Gamma_0^2(N) }}\; \;  
		\left( \Phi_{\infty} (\epsilon_{c,d}(\gamma_1 g_{1,\infty},\gamma_2 g_{2,\infty}),s)  \bigotimes_{p < \infty}  \phi_p \right)(1). \\
\end{align*}
Next using $ \eqref{matrix_identity} $, we write
\begin{align*}
\mat{a_j}{b_j}{c_j}{d_j} \bmat{1}{x_j}{}{1} \bmat{\sqrt{y_j}}{}{}{\sqrt{y_j}^{-1}} = \bmat{1}{x_j^{\prime}}{}{1} \bmat{\sqrt{y_j^{\prime}}}{}{}{\sqrt{y_j^{\prime}}^{-1}} r(\theta_j), 
\end{align*}
with 
\begin{equation*}
\bmat{a_j}{b_j}{c_j}{d_j} \in \SL(2,\ZZ), \, \tau = x_j+iy_j, \, x_j^{\prime} + iy_j^{\prime} = \frac{a_j \tau_j + b_j}{c_j \tau_j + d_j},  \text{ and } \exp(i \theta_j) = \frac{c_j \bar{\tau_j} + d_j }{|c_j \tau + d_j|}.
\end{equation*}

Next we compute $ \Phi_{\infty} (\epsilon_{c,d}(\gamma_1 g_{1,\infty},\gamma_2 g_{2,\infty}),s)$.
\begin{align*}
&\Phi_{\infty} \left(\epsilon_{c,d}(\gamma_1 g_{1,\infty},\gamma_2 g_{2,\infty}),s \right) \\
&= \Phi_{\infty} \left(\epsilon_{c,d} \left(\bmat{\sqrt{y_1^{\prime}}}{x_1^\prime \sqrt{y_1^{\prime}}^{-1}}{}{\sqrt{y_1^{\prime}}^{-1}} r(\theta_1),\bmat{\sqrt{y_2^{\prime}}}{x_2^{\prime} \sqrt{y_2^{\prime}}^{-1}}{}{\sqrt{y_2^{\prime}}^{-1}} r(\theta_2) \right),s\right)\\
&= \Phi_{\infty} \left(\epsilon_{c,d} \left(\bmat{\sqrt{y_1^{\prime}}}{x_1^{\prime} \sqrt{y_1^{\prime}}^{-1}}{}{\sqrt{y_1^{\prime}}^{-1}},\bmat{\sqrt{y_2^{\prime}}}{x_2^{\prime} \sqrt{y_2^{\prime}}^{-1}}{}{\sqrt{y_2^{\prime}}^{-1}} \right),s \right) \, e^{ik \theta_1 }  e^{ik \theta_2 }\\
&= \Phi_{\infty} \left(\epsilon_{c,d} \left(\bmat{\sqrt{y_1^{\prime}}}{x_1^{\prime} \sqrt{y_1^{\prime}}^{-1}}{}{\sqrt{y_1^{\prime}}^{-1}},\bmat{\sqrt{y_2^{\prime}}}{x_2^{\prime} \sqrt{y_2^{\prime}}^{-1}}{}{\sqrt{y_2^{\prime}}^{-1}} \right),s \right) \, \left(\frac{c_1 \bar{\tau_1} + d_1 }{|c_1 \tau_1 + d_1|}\right)^k  \left(\frac{c_2 \bar{\tau_2} + d_2 }{|c_2 \tau_2 + d_2|}\right)^k.
\end{align*}

Let $ \alpha_j = \sqrt{y^{\prime}_j} $ and $ \beta_j = x_i^{\prime} \sqrt{y_i^{\prime}}^{-1} $. Then,
\begin{align*}
& \Phi_{\infty} (\epsilon_{c,d} \left(\bmat{\sqrt{y_1^{\prime}}}{x_1^{\prime} \sqrt{y_1^{\prime}}^{-1}}{}{\sqrt{y_1^{\prime}}^{-1}},\bmat{\sqrt{y_2^{\prime}}}{x_2^{\prime} \sqrt{y_2^{\prime}}^{-1}}{}{\sqrt{y_2^{\prime}}^{-1}} \right), s) \\
& = 
\Phi_{\infty}\left(  \left[ \begin{matrix*}[r]
d &  -c &  &  \\
-b  & a  &  &  \\
&   & a & b \\
& &  c   & d
\end{matrix*} \right]   \left[ \begin{matrix*}[r]
\alpha_1 &   & \beta_1  &  \\
& \alpha_2  &  & \beta_2  \\
&   & \alpha_1^{-1} &\\
& &     & \alpha_2^{-1}
\end{matrix*} \right],s \right) \\
&	= \Phi_{\infty}\left(  \left[ \begin{matrix*}[r]
\alpha_1 d & -\alpha_{2} c & \beta_{1} d & -\beta_2 c \\
-\alpha_1 b &  \alpha_{2}a & - \beta_1 b & \beta_2 a \\
&  & \frac{a}{\alpha_1} & \frac{b}{\alpha_{2}} \\
&  & \frac{c}{\alpha_1} & \frac{d}{\alpha_{2}}
\end{matrix*} \right], s \right) \\
&	= \Phi_{\infty}\left( 
\left[ \begin{matrix*}[r]
1  &   &  &  \\
& \alpha_1 \alpha_{2} & & \\
& & 1& \\
& & & (\alpha_1 \alpha_{2})^{-1} \\
\end{matrix*} \right]  
\left[ \begin{matrix*}[r]
\alpha_1 d & -\alpha_{2} c & \beta_{1} d & -\beta_2 c \\
\frac{-b}{\alpha_2} &  \frac{a}{\alpha_1} & \frac{- \beta_1 b}{\alpha_1 \alpha_2} & \frac{\beta_2 a}{\alpha_1 \alpha_2} \\
&  & \frac{a}{\alpha_1} & \frac{b}{\alpha_{2}} \\
&  & {c}{\alpha_2} & {d}{\alpha_{1}}
\end{matrix*} \right], s \right) \\
&	= |\alpha_1 \alpha_2 |^{s+2} \,\,\Phi_{\infty}\left( \bmat{A \invtr}{B}{}{A}, s  \right),
\end{align*}
where $\ds A = \bmat{\frac{a}{\alpha_1}}{\frac{b}{\alpha_2}}{c\alpha_2}{d\alpha_1} $, $\ds B = \bmat{d\beta_1}{-c\beta_2 }{\frac{- \beta_1 b}{\alpha_1 \alpha_2}}{\frac{\beta_2 a}{\alpha_1 \alpha_2}} $.\\
Now,
\begin{align*}
\Phi_{\infty}\left( \bmat{A \invtr}{B}{}{A}  \right) &= \Phi_{\infty}\left( \bmat{1}{BA^{-1}}{}{1}\bmat{A \invtr}{}{}{A}, s  \right)\\
&= \Phi_{\infty}\left( \bmat{1}{BA^{-1}}{}{1}\bmat{X \invtr}{}{}{X} \bmat{R^{-1}}{}{}{R} \bmat{r(\theta)}{}{}{r(\theta)}, s \right),
\end{align*}
where $ X = \bmat{1}{q}{}{1}  $, $R= \bmat{r^{-1}}{}{}{r}  $,
with $\ds q = \frac{\alpha_2^2 ac +\alpha_1^2 bd}{\alpha_1 \alpha_2(\alpha_2^2 c^2 + \alpha_1^2 d^2)} $ and $ \ds r^2 = {\alpha_2^2 c^2 + \alpha_1^2 d^2}  $.\\
It follows from the definition of $ \Phi_{\infty} $ that it is invariant under the action of $ \pi(\bmat{r(\theta)}{}{}{r(\theta)})  $.
Therefore, we only need to compute
\begin{align*}
\Phi_{\infty}\left( \bmat{1}{BA^{-1}}{}{1}\bmat{X \invtr}{}{}{X} \bmat{R^{-1}}{}{}{R} ,s \right) &= \Phi_{\infty}\left( \bmat{X \invtr}{}{}{X} \bmat{1}{X^{T} BA^{-1}X }{}{1} \bmat{R^{-1}}{}{}{R}, s  \right) \\
& = \Phi_{\infty}\left( \bmat{1}{X^{T} BA^{-1}X }{}{1} \bmat{R^{-1}}{}{}{R}, s  \right).
\end{align*}
A simple calculation yields 
\begin{align*}
X^{T} BA^{-1}X = \left[\begin{array}{rr}
\alpha_{2} \beta_2 c^{2} + \alpha_1 \beta_1 d^{2} & * \\
* & *
\end{array}\right].
\end{align*}
Therefore,
\begin{align*}
\Phi_{\infty}\left( \bmat{1}{X^{T} BA^{-1}X }{}{1} \bmat{R^{-1}}{}{}{R},s  \right)
& =  |r|^{-(s+2)}\pi_{\infty}\left( \bmat{1}{\alpha_{2} \beta_2 c^{2} + \alpha_1 \beta_1 d^{2}}{}{1} \bmat{r}{}{}{r^{-1}} \right) \Phi_{\infty}(1,s) \\
& = |r|^{-(s+2)} \pi_{\infty}\left( \bmat{1}{\alpha_{2} \beta_2 c^{2} + \alpha_1 \beta_1 d^{2}}{}{1} \bmat{r}{}{}{r^{-1}} \right) \phi_{\infty}.
\end{align*}	
Now we obtain,
\begin{align*}
I(T_r)  & = \sum_{\substack{ c \in \NN, d \in \ZZ\backslash \{0\} \\ (c,d) =1}} \sum_{\substack{  \gamma_i  = \mat{a_i}{b_i}{c_i}{d_i} \in  \Gamma_{\infty} \backslash \Gamma_0^2(N) }}\; \;  
\left( \Phi_{\infty} (\epsilon_{c,d}(\gamma_1 g_{1,\infty},\gamma_2 g_{2,\infty}),s)  \bigotimes_{p < \infty}  \phi_p \right)(1) \\
& = \sum_{\substack{ c \in \NN, d \in \ZZ\backslash \{0\} \\ (c,d) =1}} \sum_{\substack{  \gamma_i  = \mat{a_i}{b_i}{c_i}{d_i} \in  \Gamma_{\infty} \backslash \Gamma_0^2(N) }}\; \;  
|\alpha_1 \alpha_2 |^{s+2} |r|^{-(s+2)} \left(\frac{c_1 \bar{\tau_1} + d_1 }{|c_1 \tau_1 + d_1|}\right)^k  \left(\frac{c_2 \bar{\tau_2} + d_2 }{|c_2 \tau_2 + d_2|}\right)^k  \\ 
& \qquad \qquad \qquad \qquad \qquad \qquad \qquad \pi_{\infty}\left( \bmat{1}{\alpha_{2} \beta_2 c^{2} + \alpha_1 \beta_1 d^{2}}{}{1} \bmat{r}{}{}{r^{-1}} \right) \phi_{1}(1) \\
& = \sum_{\substack{ c \in \NN, d \in \ZZ\backslash \{0\} \\ (c,d) =1}} \sum_{\substack{  \gamma_i  = \mat{a_i}{b_i}{c_i}{d_i} \in  \Gamma_{\infty} \backslash \Gamma_0^2(N) }}\; \;  
(y_1 y_2)^{\frac{s+2}{2}}	j(\gamma_1,\tau_1)^{-(s+2)} j(\gamma_2,\tau_2)^{-(s+2)}  \\ 
& \qquad \qquad \qquad \qquad \qquad \qquad   |r|^{-(s+2)} \phi_1 \left( \bmat{1}{\alpha_{2} \beta_2 c^{2} + \alpha_1 \beta_1 d^{2}}{}{1} \bmat{r}{}{}{r^{-1}} \right)  \\
& = \sum_{\substack{ c \in \NN, d \in \ZZ\backslash \{0\} \\ (c,d) =1}} \sum_{\substack{  \gamma_i  = \mat{a_i}{b_i}{c_i}{d_i} \in  \Gamma_{\infty} \backslash \Gamma_0^2(N) }}\; \;  
(y_1 y_2)^{\frac{s+2}{2}}	\, \, j(\gamma_1,\tau_1)^{-k} j(\gamma_2,\tau_2)^{-k}   f\left( c^2 \gamma_2 \langle\tau_2 \rangle + d^2 \gamma_1 \langle \tau_1 \rangle \right) \\
& = 2 \sum_{\substack{ c, d \, \in \NN,  \\ (c,d) =1}} \sum_{\substack{  \gamma_i  = \mat{a_i}{b_i}{c_i}{d_i} \in  \Gamma_{\infty} \backslash \Gamma_0^2(N) }}\; \;  
(y_1 y_2)^{\frac{s+2}{2}}	\, \, j(\gamma_1,\tau_1)^{-k} j(\gamma_2,\tau_2)^{-k}   f\left( c^2 \gamma_2 \langle\tau_2 \rangle + d^2 \gamma_1 \langle \tau_1 \rangle \right). 
\end{align*}
This completes the proof.
\end{proof}

\section{Pullback of Klingen Eisenstein series with respect to Paramodular subgroup} \label{Sect:paramodular}
In this section we obtain a pullback formula for Klingen Eisenstein series with respect to paramodular subgroup. The proof essentially proceeds along the line of the proof of pullback of  Klingen Eisenstein series with respect to Siegel congruence subgroup, presented earlier. The difference here is mainly in the selection of appropriate local vectors $ \Phi_p $. We note an interesting feature of Klingen Eisenstein series with respect to paramodular subgroup considered in \cite{shukla2018klingen}. Indeed, if one starts with level $ N $ elliptic cusp form $ f $ then the corresponding Klingen Eisenstein series obtained using the paramodular lift was of level $ N^2 $ (See Theorem 6.2.1, \cite{shukla2018klingen}). Unsurprisingly, the same feature is reflected in the pullback result for paramodular Klingen lift that we obtain in this paper.

\mainpullbackparamodular*

\begin{proof}
	We begin by explicitly defining the global distinguished vector $ {\Phi} $.  
	Let $ \Pi $ denote the automorphic representation  $  \chi^{-1} \, |\cdot|^{s} \rtimes |\cdot|^{\frac{-s}{2}} \pi $  of 
	$ \GSp(4,\AA) $. We know from the tensor product theorem that  
	\begin{align}
	\Pi \cong \bigotimes_{p\leq \infty} \Pi_p
	\end{align}
	where almost all $ \Pi_p $ are unramified. For each prime $ p $, we pick a local distinguished vector
	$ \Phi_p \in \Pi_p $, which would then yield a global distinguished vector 
	\begin{equation}{\label{defglobaldistinguishedvectorpara1}}
	{\Phi} \cong {\Phi}_\infty \otimes \bigotimes_{p < \infty}  {\Phi}_p   \, .
	\end{equation}
	Next, we describe our choices for the local distinguished vectors. \\ 
	{\textbf{Archimedean distinguished vector}}\\ 
	We pick a distinguished vector ${\Phi}_{\infty} $ such that 
	it is of the minimal $ K $-type $ (k,k) $. 
	More explicitly we define  $ \Phi_{\infty} $ as follows
	\begin{equation}{\label{defarchimedeandistinguishedvectorpara2}}
	{\Phi}_{\infty}(h_{\infty} k_{\infty}) :=  \det(j(k_\infty,I))^{-k} |t^2(ad -bc)^{-1}|_{} \,|t|_{}^{s}\, \chi^{-1}(t)\, |ad -bc|_{}^{-\frac{s}{2}} \pi_{\infty}(\bmat{a}{b}{c}{d}) \,\phi_{\infty}
	\end{equation}
	where $$ h_{\infty} = \left[ \begin{matrix*}[r]
	a &   & b & * \\
	* &t & * & * \\
	c  &   & d & * \\
	& &   & t^{-1}(ad -bc)
	\end{matrix*}  \right] \in Q(\RR) $$ and $ k_{\infty}  \in \Kfun_1$, the maximal standard compact subgroup of $ \GSp(4,\RR) $. It can be checked that ${\Phi}_{\infty} $ is well-defined.\\ 
	\noindent \textbf{Unramified non-archimedean distinguished vectors} \\ 
	For all primes $ q $ such that $ q \nmid N $, we pick unramified local distinguished vectors such that,
	\begin{equation}{\label{defglobaldistinguishedvectorpara2}}
	{\Phi}_q(1) := \phi_q \, .
	\end{equation}    

\noindent \textbf{Ramified non-archimedean distinguished vectors} \\
For each finite prime $ p $ such that $ p | N $, we pick a $ \para_p(p^{n_p}) $ invariant vector as distinguished vector.  We note that the existence of such a vector, and that it is supported only on $ Q(Q_p ) L_{n_p} \para_p(p^{2 n_p}) $, follows from the proof of the Theorem $ 5.4.2 $ in [30]. So,
our distinguished vector is given by
\begin{align}{\label{DeflocalcongruencedistinguishedvectorPrincipalSeriesPara}}
	{\Phi}_{p}  \left( \left[ \begin{matrix*}[r]
	a &   & b & * \\
	* &t & * & * \\
	c  &   & d & * \\
	& &   & t^{-1}(ad -bc)
	\end{matrix*}  \right]  L_{n_p} \,  \kappa   \right)  = &|t^2(ad -bc)^{-1}|_{p} \,|t|_{p}^{s}\, \chi_p^{-1}(t) \, |ad -bc|_{p}^{-\frac{s}{2}} \pi_{p}(\bmat{a}{b}{c}{d}) \phi_{p},
\end{align}  where, $ \kappa \in \para_p({p^{2n_p}})$,$ \bmat{a}{b}{c}{d} \in \GL(2,\QQ_{p}) $  and  $ \phi_{p} \in   \pi_{p}$ is a local newform of level $ n_p $.
	Now we compute the contribution of the three terms $ T_1 $, $ T_{s_1} $ and $ T_r $ in Lemma \ref{Lem:Heim}. \\
	We have
	\begin{align*}
	I(T_1) &= \sum_{\substack{\gamma_1 = \, (1,\gamma), \, \gamma \in \Gamma_{\infty} \backslash \SL(2,\ZZ )}}\; (\Phi (\gamma_1 g,s))(1) \\
    &= \sum_{\substack{  \gamma \in\Gamma_{\infty}  \backslash \SL(2,\ZZ )}}\; \; (\Phi (  ( g_1,\gamma g_2),s))(1) \\
	&=	\sum_{\substack{  \gamma \in \Gamma_{\infty}  \backslash \SL(2,\ZZ )   }}\; \; (\Phi (  ( g_1,\gamma g_2),s))(1)\\
	& = \sum_{\substack{  \gamma \in \Gamma_{\infty}  \backslash \SL(2,\ZZ)   }}\; \;(|\det(g_1)|^{-\frac{s+2}{2}}  \pi(g_1) \Phi ( (1,(\gamma g_{2,\infty},\gamma,\cdots) ),s))(1) \\
	& = \sum_{\substack{  \gamma  \in \Gamma_{\infty}  \backslash \SL(2,\ZZ)   }}\; \pi(g_1) \
	\left( \Phi_{\infty} ((1,\gamma g_{2,\infty}),s)  \bigotimes_{p | N}  \Phi_p((1,\gamma),s) \bigotimes_{q < \infty, \, q \nmid   N} \Phi_q((1,\gamma),s)\right)(1).
	\end{align*}
	Since $ \Phi_p $ is supported only on $ Q(\QQ_p) L_{p^{n_p}} \para_p({p^{2n_p}}) $, using Lemma 5.1, \cite{schmidt2018klingen}, we conclude that only those $ \gamma $ will contribute in the summation for which 
	$$ (1,\gamma) \in \Q(\QQ) \backslash   \bigcap_{p | N}  Q(\QQ) \,L_{p^{n_p}} \Kfun({p^{2n_p}}) = \Q(\QQ) \backslash   Q(\QQ) \,L_N\, \Kfun(N^2) = L_N \, D(N),$$ 
	where $D(N)$ is a set of representatives for $\Lfun_N^{-1}Q(\Q)\Lfun_N\cap\Kfun(N^2)\backslash\Kfun(N^2)$. A calculation shows that  $ \gamma \in \Gamma_{\infty} \backslash \Gamma_0^2(N^2) $. 
	Further, if $ q \nmid N $ then $ \Phi_q((1,\gamma),s) = \Phi_q(1,s) = \phi_q $, as  $ \Phi_q $ is unramified and $ (1,\gamma) \in \GSp(4, \ZZ_q ) $.  Also, for $ p|N $, we get from the definition of $ \Phi_p $ that, $ \Phi_p((1,\gamma),s) = \Phi_p(1,s) = \phi_p  $. Therefore, the above discussion allows us to write
	\begin{align*}
	I(T_1)  & = \sum_{\substack{  \gamma  = \mat{a}{b}{c}{d} \in  \Gamma_{\infty} \backslash \Gamma_0^2(N^2) }}\; \;  
\pi(g_1)	\left( \Phi_{\infty} ((1,\gamma g_{2,\infty}),s)  \bigotimes_{p < \infty}  \phi_p \right)(1). \\
	\end{align*}
	At this stage a simple calculation along the line of Theorem \ref{thm:main_pullback}, whose details we omit, shows that
	\begin{align*}
		I(T_1)	& = (y_1 y_2)^{\frac{s+2}{2}} {E_1}(s,k,N^2,\tau_2) \,  f(\tau_1). 
	\end{align*}
	\underline{Contribution of $ T_{s_1} $:} A calculation, similar to the previous one, shows that the contribution of  $ T_{s_1} $ term is given by
	\begin{align*}
	I(T_{s_1})	& = (y_1 y_2)^{\frac{s+2}{2}} E_1(s,k,N^2,\tau_1) \,  f(\tau_2). 
	\end{align*}
	\underline{Contribution of $ T_{r} $:} Now we find the contribution of $ T_{r} $.
	\begin{align*}
	& I(T_r) = \sum_{\substack{ \gamma_1, \gamma_2 \in \Gamma_{\infty} \backslash \SL(2,\ZZ ) \\  c \in \NN, d \in \ZZ\backslash \{0\} \\ (c,d) =1}} (\Phi (\epsilon_{c,d}  (\gamma_1 g_1,\gamma_2 g_2),s))(1) \\
		& = \sum_{\substack{  \substack{ \gamma_1, \gamma_2 \in \Gamma_{\infty} \backslash \SL(2,\ZZ ) \\  c \in \NN, d \in \ZZ\backslash \{0\} \\ (c,d) =1}}}\; 
	\left( \Phi_{\infty} (\epsilon_{c,d} (\gamma_1 g_{1,\infty},\gamma_2 g_{2,\infty}),s)  \bigotimes_{p | N}  \Phi_p((\gamma_1,\gamma_2),s) \bigotimes_{q < \infty, \, q \nmid   N} \Phi_q((\gamma_1,\gamma_2),s)\right)(1).
	\end{align*}
	Similar to the calculation for the contribution of $ T_1 $ case, only those $ (\gamma_1,\gamma_2) $ will contribute in the summation for which 
	$$ (\gamma_1,\gamma_2) \in \Q(\QQ) \backslash   \bigcap_{p | N}  Q(\QQ) \,L_{p^{n_p}} \Kfun({p^{2n_p}}) = \Q(\QQ) \backslash   Q(\QQ) \,L_N\, \Kfun(N^2) = L_N \, D(N).$$
	This implies, on performing a simple calculation that $ \gamma_1 \in \Gamma_{\infty}  \backslash \SL(2,\ZZ)$ and $\gamma_2 \in \Gamma_{\infty}  \backslash \Gamma_0^2(N^2) $.
At this point, a calculation exactly similar to Theorem \ref{thm:main_pullback} shows that
	\begin{align*}
	I(T_r)  & = \sum_{\substack{c \in \NN, d \in \ZZ\backslash \{0\}, \\ (c,d)= 1}}  \sum_{\substack{\gamma_1 \in  \Gamma_{\infty} \backslash \SL(2,\ZZ), \\ \gamma_2 \in  \Gamma_{\infty} \backslash \Gamma_0^2(N^2)}}  j(\gamma_1,\tau_1)^{-(s+2)} j(\gamma_2,\tau_2)^{-(s+2)}  f(d^2 \gamma_1 \langle \tau_1 \rangle + c^2 \gamma_2  \langle \tau_2 \rangle).
	\end{align*}
	This completes the proof.
\end{proof}

\section{Some applications}\label{Sect:applications}

Now, we give some applications of the pullback formula for the classical Klingen Eisenstein series that we described earlier. The key idea is to compare the Fourier coefficients obtained from the Fourier series expansion of the Klingen Eisenstein series in degree two given by Shin-ichiro Mizumoto in \cite{MR636883} and \cite{MR733590}, with the Fourier coefficients  obtained by using the pullback formula derived earlier (also see Remark~$ \ref{remark:cor} $). Here we will only consider the full level case, i.e.,$ N=1 $, but the method obviously works for any general level $ N $ if formulas for  Fourier series expansion of the Klingen Eisenstein series with level with respect to Siegel and paramodular congruence subgroups are known. 
We also note that 
for $s=k-2$, and $ N=1 $ the Eisenstein series defined by $\eqref{def:Eisensteinseriesmodular}$ reduces to a classical homomorphic Eisenstein series, i.e.,
\begin{align}
E_1(k-2,k,z,1) =   E_k(z) =  \sum_{\substack{    \mat{a}{b}{c}{d} \in \Gamma_{\infty} \backslash \SL(2,\ZZ)    }}\; \; \frac{1}{(c \tau +d)^k},
\end{align}   
where $ E_k(z)  $ is the normalized weight $ k $ holomorphic Eisenstein series with the following Fourier expansion
\begin{align} \label{def:holomorphic_normalized_Eisenstein_series}
E_k(z) = 1 + \frac{2}{\zeta(1-k)} \sum_{m \geq 1} \sigma_{k-1}(m) q^m.
\end{align}
For non-negative numbers $ n_1,n_2 $, let
\begin{align}
\Lambda (n_1,n_2) := \left\{ \left(\begin{smallmatrix} n_1& b/2 \\ b/2 &n_2 \end{smallmatrix} \right)  \vert \, b \in \ZZ, 4n_1 n_2 - b^2 \geq 0 \right\}.
\end{align}
Then, for $ T = \mat{n_1}{\frac{b}{2}}{\frac{b}{2}}{n_2} $ and $ Z = \mat{\tau_1}{}{}{\tau_2} $ we have 
\begin{align} \label{eq:fourier_expansion_Klingen}
E_{2,1}^{k}(Z,f)  = \sum_{\substack{T \in \Lambda (n_1,n_2) \\ n_1, n_2 \geq 0} } A(T,f) q_1^{n_1} q_2^{n_2} ,
\end{align}
where we write $ q_1 = e^{2 \pi i \tau_1} $ and $ q_2 = e^{2 \pi i \tau_2} $.

Further, for $T$ such that $-\det(2T)$ is a fundamental discriminant, $ A(T,f) $, is given by Theorem 1 in \cite{MR636883} as follows
\begin{equation} \label{Eq:Fourier_fundamental_discr}
A(T,f) = (-1)^{\frac{k}{2}} \frac{(k-1)!}{(2k-2)!}(2 \pi)^{k-1} \det(2T)^{k- \frac{3}{2}} \frac{L(k-1, \chi_{-\det(2T)}) L(k-1,f \otimes \vartheta_T) }{L(2k-2,\Sym^2 f)}
\end{equation}
with $ \chi_{-\det(2T)} $ being the Dirichlet character associated to the field $\mathbb{Q}(\sqrt{-\det(2T)})$. 

Now we give a proof of \Cref{thm:main_identity}.
\mainidentity*
\begin{proof}
	The theorem will follow by comparing the $ q_1 q_2 $ coefficients of Fourier expansions on both sides of $ \eqref{Eq:main_pullback_formula}$.
	Let $ f(z)  $ be a normalized elliptic cuspform with the following Fourier expansion
	\begin{align*}
	f(z) = \sum_{n=1}^{\infty} \, a(n) \, q^n.
	\end{align*}
	Here $ f $ is normalized means $ a(1) =1 $.
	For comparing the $ q_1 q_2 $ coefficients using the Fourier expansion of $ E_{2,1}^{k}(Z) $  given by $ \eqref{eq:fourier_expansion_Klingen} $,
	 we must have $ n_1 = n_2 =1 $. Then, since $n_1 = n_2 =1$, the only possible values for $b$ are $0, \pm 1, \pm 2$. Therefore, the $q_1 q_2$-coefficient of $E_{2,1}^{k}(Z)$ is given by
	\begin{equation*}
	\sum_{b=-2}^{2} A \left( \mat{1}{b/2}{b/2}{1},f \right).
	\end{equation*} 
	Since $ \lim\limits_{y \to \infty}(E_{2,1}^k(\mat{\tau}{}{}{iy} ; f)) = f(\tau) $  it follows that 
	\begin{align*}
	A \left( \mat{1}{0}{0}{0}, f \right) = 1.
	\end{align*}
    Moreover, it can easily be verified that 
	$ \mat{1}{1}{1}{1} $ and $ \mat{1}{-1}{-1}{1} $ are both unimodularly equivalent to $ \mat{1}{0}{0}{0} $, therefore,
	\begin{equation*} 
	A \left( \mat{1}{1}{1}{1}, f \right) = A \left( \mat{1}{-1}{-1}{1}, f \right) =  1.
	\end{equation*}
    For each of the remaining three values of $T$, $-\det(2T)$ is a fundamental discriminant. Therefore, from $ \eqref{Eq:Fourier_fundamental_discr} $  we obtain the following results.
	\begin{equation*}
	A \left(\mat{1}{}{}{1}, f \right) = (-1)^{\frac{k}{2}} \frac{(k-1)!}{(2k-2)!}(2 \pi)^{k-1} 2^{2k- 3} \frac{L(k-1, \chi_{-4}) L(k-1,f \otimes \vartheta_1) }{L(2k-2,\Sym^2 f)},
	\end{equation*}
    and
	\begin{align*}
	A \left(\mat{1}{-\frac{1}{2}}{-\frac{1}{2}}{1}, f \right) &= A \left( \mat{1}{\frac{1}{2}}{\frac{1}{2}}{1}, f \right) \\
	&= (-1)^{\frac{k}{2}} \frac{(k-1)!}{(2k-2)!}(2 \pi)^{k-1}  3^{k- \frac{3}{2}} \frac{L(k-1, \chi_{-3}) L(k-1,f \otimes \vartheta_2)}{L(2k-2,\Sym^2 f)}.
	\end{align*}
	Therefore, we get that 
	\begin{align} \label{Eq:right}
	&\sum_{b=-2}^{2} A \left( \mat{1}{b/2}{b/2}{1},f \right) =  \nonumber \\&  2 +\frac{(-1)^{k/2}(k-1)!(2\pi)^{k-1}}{(2k-2)!L(2k-2,\Sym^2 f)}  \left[2^{2k-3}L(k-1,\chi_{-4})L(k-1,f \otimes \vartheta_1) \right.  \nonumber \\ &\left.  + 2 \cdot 3^{k-3/2}L(k-1,\chi_{-3})L(k-1,f \otimes \vartheta_2)\right]. 
	\end{align}
	Now since $ a(1) =1 $, and from $ \eqref{def:holomorphic_normalized_Eisenstein_series} $ we see that the coefficient of $ q $ in the Fourier-series expansion of $ E_k(z) $ is $ \frac{2}{\zeta(1-k)} $, it follows that the $ q_1q_2 $-coefficient in the Fourier-series expansion of 
	$ E_k (\tau_1) f(\tau_2) + E_k (\tau_2) f(\tau_1 ) $ is given by 
	\begin{align}\label{Eq:left}
	\frac{4}{\zeta(1-k)}.
	\end{align}
The $ q_1q_2 $-coefficient in the Fourier-series expansion of the remaining term on the right side is given by (see $ \eqref{Eq:defA_f} $)
\begin{align} \label{Eq:left2}
 A_f(1,1).
\end{align}
Now the theorem follows from $ \eqref{Eq:right} $, $ \eqref{Eq:left} $ and $ \eqref{Eq:left2} $.
\end{proof}
\begin{remark}\leavevmode
	\begin{enumerate}
		\item 	A similar expression as given in the right side of Eq~$\eqref{eq:theorem1.1}$, appears in the weighted average formula for critical $ L $-values given in Theorem $ 1.1 $ in \cite{amersi2012pullbacks}. However, the correct definition of  
		$ \mc{A}_{k}(f) $ in Theorem $ 1.1 $, \cite{amersi2012pullbacks}, should be 
		\begin{align} \label{Eq:Akf_right}
		\mc{A}_{k}(f) &=  \zeta(k-1) \biggl( 2 +\frac{(-1)^{k/2}(k-1)!(2\pi)^{k-1}}{(2k-2)!L(2k-2,\Sym^2 f)}  \left[2^{2k-3}L(k-1,\chi_{-4})\right. \nonumber \\ & \left. L(k-1,f \otimes \vartheta_1) \right. \biggr.  \biggl. \left.  + 2 \cdot 3^{k-3/2}L(k-1,\chi_{-3})L(k-1,f \otimes \vartheta_2)\right] \biggr). 
		\end{align}
		\item Clearly $ \mc{A}_{k}(f) $ as defined in $ \eqref{Eq:Akf_right} $ reduces to $ \zeta(k-1) (\frac{4 }{\zeta(1-k)} + A_f(1,1)) $ by Theorem~$ \ref{thm:main_identity} $. It can be used to further simplify the result in Theorem $ 1.1 $, \cite{amersi2012pullbacks}.	
	\end{enumerate}	
\end{remark}
Finally, we give the proof of \Cref{thm:main_theorem}.
\begin{proof}[\textbf{Proof of \Cref{thm:main_theorem}}]
Since $ \gcd(n_1,n_2) =1 $, Theorem~$ 1 $ in \cite{MR733590} is applicable, and therefore, the result follows on comparing the $ q_1^{n_1} q_2^{n_2} $ coefficients of Fourier-series expansions on both sides of $ \eqref{eq:fourier_expansion_Klingen}$, and thereafter doing some simple algebraic manipulations.
\end{proof}

\bibliographystyle{plain}

\end{document}